\documentclass[11pt]{amsart}

\usepackage{amsmath, amsthm, amssymb}
\usepackage{verbatim}
\usepackage{tikz}

\usepackage{url}
\usepackage{hyperref}

\renewcommand{\d}{\partial}

\newcommand{\vepsilon}{\varepsilon}
\newcommand{\vphi}{\varphi}

\newcommand{\cE}{\mathcal{E}}

\newcommand{\cM}{\mathcal{M}}

\newcommand{\supp}{\mbox{supp }}

\newcommand{\bC}{\mathbb{C}}

\newtheorem{thm}{Theorem}
\newtheorem{prop}[thm]{Proposition}
\newtheorem{lem}[thm]{Lemma}
\newtheorem{cor}[thm]{Corollary}

\theoremstyle{definition}
\newtheorem{defn}[thm]{Definition}
\newtheorem{remark}[thm]{Remark}

\newtheorem{prob}[thm]{Problem}

\numberwithin{thm}{section}
\numberwithin{equation}{section}

\renewcommand{\[}{\begin{equation}}
\renewcommand{\]}{\end{equation}}

\newcommand{\wed}{\wedge}

\title[H\"older continuous subsolution problem]{On the H\"older continuous subsolution problem for the complex Monge-Amp\`ere equation}

\author[N.-C Nguyen]{Ngoc Cuong Nguyen}

\address{Department of Mathematics and Center for Geometry and its Applications, Pohang University of Science and Technology, 37673, The Republic of Korea}
\email{cuongnn@postech.ac.kr}


\begin{document}

\maketitle


\em Dedicated to Professor Kang-Tae Kim on the occasion of his $60$th birthday \rm

\bigskip

\bigskip

\bigskip

\begin{abstract} We give a necessary and sufficient condition for positive Borel measures such that the Dirichlet problem, with zero boundary  data, for the complex Monge-Amp\`ere equation  admits H\"older continuous plurisubharmonic solutions. In particular, when the subsolution has finite Monge-Amp\`ere total mass, we obtain an affirmative answer to a question of Zeriahi \cite{DGZ16}.\end{abstract}

\section{Introduction}

Bedford and Taylor \cite{BT76} proved the existence of plurisubharmonic solutions of the Dirichlet problem for the complex Monge-Amp\`ere equation in a strictly pseudoconvex bounded domain $\Omega$ in $\bC^n$. The solution is continuous or H\"older continuous provided that the data is continuous or H\"older continuous, respectively. In their subsequent fundamental paper \cite{BT82} they developed pluripotential theory which becomes a powerful tool to study plurisubharmonic functions. Later, applying pluripotential theory methods the weak solution theory has been developed much further by Cegrell \cite{Ce98, Ce04} and Ko\l odziej \cite{ko98,ko05}. The continuous solution is obtained for more general right hand sides in \cite{ko05}, in particular for $L^p-$density measures, $p>1$, as an important class. Recently, Guedj, Zeriahi and Ko\l odziej \cite{GKZ08} showed that the solution is H\"older continuous, which is the optimal regularity, for such measures,  under some extra assumptions. Finally, the extra assumptions were removed  in \cite{BKPZ16}, \cite{Cha15a}. However, there is still an open question to find a characterisation for the measures admitting H\"older continuous solutions to the equation. If one requires only bounded solutions, then Ko\l odziej's subsolution theorem \cite{ko95} gives such a criterion. 

Let $\vphi \in PSH(\Omega)\cap C^{0,\alpha}(\bar \Omega)$ with $0 <\alpha \leq 1$, where $C^{0,\alpha}(\bar\Omega)$ stands for the set of H\"older continuous functions of exponent $\alpha$ in $\bar\Omega$. Moreover, we assume that 
\[\label{eq:sub-zero-boundary}
	\vphi = 0 \quad \mbox{on } \d \Omega.
\]
Given such a function $\vphi$ we consider the set of positive Borel measures on $\Omega$ which are dominated by the Monge-Amp\`ere operator of $\vphi$, namely,
\[\label{eq:holder-sub-1} \cM (\vphi, \Omega):= \left\{
\mu \mbox{ is positive Borel measure: } \mu \leq (dd^c\vphi)^n \mbox{ in } \Omega
\right\}.
\]
For measures in this set we also say that $\vphi$ is a H\"older continuous {\em subsolution} to $\mu$.
Let $\mu \in \cM (\vphi,\Omega)$ and consider the Dirichlet problem for $0< \alpha' \leq 1$,
\[\label{eq:zeriahi-prob}\begin{aligned}
&	u \in PSH(\Omega) \cap C^{0,\alpha'}(\bar \Omega)\\
&	(dd^cu)^n = \mu, \\
&	u=0 \quad\mbox{on } \d\Omega.
\end{aligned}\]
The following problem \cite[Question 17]{DGZ16} was raised by Zeriahi.  
\begin{prob} \label{prob:zeriahi} 
Can we always solve the Dirhichlet problem~\eqref{eq:zeriahi-prob}  for some $0<\alpha' \leq 1$?
\end{prob}
When $\vphi$ is merely bounded subsolution, the subsolution theorem in \cite{ko95} has provided a unique bounded solution.  Thus, to answer Zeriahi's question it remains to show the H\"older continuity of the bounded solution. Our main result is the following:

\begin{thm}
\label{thm:main} 
Let $\mu\in \cM(\vphi,\Omega)$. Assume that the H\"older continuous subsolution $\vphi$ has  finite Monge-Amp\`ere mass on $\Omega$, i.e.,
\[\label{eq:holder-sub-2}
	\int_\Omega (dd^c\vphi)^n <+\infty.
\] Then, the Dirichlet problem \eqref{eq:zeriahi-prob} is solvable.
\end{thm}

We actually obtain a necessary and sufficient condition for a measure in $\cM(\vphi,\Omega)$ with finite total mass such that the Dirichlet problem~\eqref{eq:zeriahi-prob} is solvable. Using this characterisation (Theorem~\ref{thm:general}) we can reprove the results from \cite{BKPZ16}, \cite{Cha15a}, \cite{GKZ08} in the case of zero boundary. We also show in Corollary~\ref{cor:simple-consequence} that there are several class of measures which satisfy the assumptions of the theorem. Charabati \cite{Cha15b} has studied very recently Problem~\ref{prob:zeriahi} for these measures. 

There is a strong connection between Problem~\ref{prob:zeriahi} and the H\"older continuity of weak solutions to Monge-Amp\`ere equations on a compact K\"ahler manifold $(X,\omega)$. Namely, the characterisation \cite[Theorem~4.3]{DDGKPZ14} (see also \cite{ko08}) says that a positive Borel measure on $(X,\omega)$ is the Mong-Amp\`ere measure of a H\"older continuous $\omega-$plurisubharmonic function if and only if  it is dominated locally (on local coordinate charts) by the one of H\"older continuous plurisubharmonic functions. On the other hand, Dinh-Nguyen \cite{DN16} has found another global characterisation for this problem by using super-potential theory \cite{DS09}. This characterisation has been used in \cite{viet16} to generalise the result obtained previously by Hiep \cite{hiep}. Our necessary and sufficient condition can be consider as the local analogue of \cite{DN16}. It may possibly be used to give a local proof of the results in \cite{hiep}, \cite{viet16}.

\medskip

{\em Acknowledgement.}  I am very grateful to S\l awomir Ko\l odziej for giving many valuable comments on the drafts of the paper which helped to improve significantly its final version. The author is supported by  the NRF Grant 2011-0030044 (SRC-GAIA) of The Republic of Korea.

\section{A general characterisation}

Let $\Omega$ be a strictly pseudoconvex bounded domain in $\bC^n$. Let $\rho\in C^2(\bar\Omega)$ be a strictly plurisubhamonic defining function for $\Omega = \{\rho<0\}$.  The standard K\"ahler form in $\bC^n$ is denoted by $\beta:=dd^c |z|^2$. Without of loss generality we may assume that 
\[\label{eq:def-fct-a}
	dd^c \rho \geq \beta \quad \mbox{ on } \bar\Omega.
\]
In this section we will prove a general characterisation of measures in $\cM(\vphi,\Omega)$ which are Mong-Amp\`ere measures of H\"older continuous plurisubharmonic functions. To state our result we need definitions and properties relate to the Cegrell classes. The class $\cE_0:= \cE_0(\Omega)$ is defined by
\[ \left\{v\in PSH\cap L^\infty(\Omega) : \lim_{z\to \d \Omega} v(z) =0, \mbox{ and } \int_\Omega (dd^c v)^n<+\infty\right\}.
\]
We will also work with a subclass of $\cE_0$, namely,
\[
	\cE_0' = \left\{v \in \cE_0 : \int_\Omega (dd^c v)^n \leq 1\right\}.
\]
The following basic properties of $\cE_0$ and $\cE_0'$ which will be used later, and we refer the readers to \cite{Ce98,Ce04} for detailed proofs.
\begin{prop}
\label{prop:cegrell-results} We have
\begin{enumerate}
\item[(a)]
The integration by parts holds true in $\cE_0$; if $u, v \in \cE_0$, so does $u+v$.
\item[(b)]
Let $u \in \cE_0$. Let $v \in PSH\cap L^\infty(\Omega)$ be such that $\lim_{z\to \d \Omega} v(z) =0$, and $v\geq u$. Then, $v\in \cE_0$ and
\[\label{eq:ma-mass-comp}
	\int_\Omega (dd^c v)^n \leq \int_\Omega (dd^cu)^n.
\] 
\item[(c)]
Let $u, v\in \cE_0'$. Then, $\max\{u,v\} \in \cE_0'$ and $\max\{u, -s\} \in \cE_0'$ for every $s>0$.
\item[(d)] For $u_1, ..., u_n\in \cE_0$ the Cegrell inequality reads as follows.
\[\label{eq:cegrell-ineq}
	\int_\Omega dd^c u_1 \wed \cdots \wed dd^c u_n 
\leq	\left(\int_\Omega (dd^cu_1)^n\right)^\frac{1}{n} \cdots \left(\int_\Omega (dd^cu_n)^n\right)^\frac{1}{n}.
\]
\end{enumerate}
\end{prop}

In what follows we always denote for $p>0$
\[
	\|\cdot\|_p := \left(\int_\Omega |\cdot|^p dV_{2n}\right)^\frac{1}{p} \mbox{ and }
	\|\cdot\|_\infty := \sup_\Omega |\cdot|,
\]
where $dV_{2n}$ is the Lebesgue measure in $\bC^n$.
The constant $C>0$ will appear in many places below, we understand that it is a uniform constant and it may differ from place to place. 

\begin{lem}
\label{lem:holder-equivalence} Let $\nu$ be a positive Borel measure on $\Omega$ and let $0<\alpha \leq 1$. Then, 
\[\label{eq:measure-holder}
	\left|\int_\Omega (u-v) d\nu \right| \leq C \|u-v\|_{1}^\alpha \quad \forall u,v \in \cE_0'
\]
if and only if 
\[ \label{eq:measure-holder-b}
\int_\Omega |u-v| d\nu \leq C \|u-v\|_{1}^\alpha \quad \forall u,v \in \cE_0'.
\]
Here, the constant $C>0$ is independent of $u,v.$
\end{lem}

\begin{proof} The sufficient condition is clear. To prove the necessary condition we first have  that 
\[\notag
	|u-v| = (\max\{u,v\} -u) + (\max\{u,v\} -v).
\]
Since $u,v\in \cE_0'$, so  $\max\{u,v\} \in \cE_0'$. Hence we apply \eqref{eq:measure-holder} twice to get the desired inequality.
\end{proof}

For a general positive Borel measure $\nu$ it defines a natural functional  
\[\notag
	\nu (w) := \int_\Omega w \;d\nu
\]
for a measurable function $w$ in $\Omega$.
We introduce the following notion which is probably a local counterpart of the one in Dinh-Sibony \cite{DS09} and Dinh-Nguyen \cite{DN16} for positive measures.

\begin{defn}
\label{def:measure-holder} The measure $\nu$ is called to be H\"older (or $\alpha-$H\"older) continuous on $\cE_0'$ if there exists $0<\alpha \leq 1$ such that
\[\label{eq:measure-holder-c}
	\left|\nu(u-v)\right| = \left| \int_\Omega (u-v) d\nu\right|\leq C \|u-v\|_1^\alpha \quad \forall u, v \in \cE_0',
\]
where the constant $C>0$ is independent of $u,v.$
\end{defn}

Thanks to Lemma~\ref{lem:holder-equivalence} we can verify the H\"older continuity of a positive Borel measure $\nu$ by using either \eqref{eq:measure-holder} or \eqref{eq:measure-holder-b}. The following properties are consequences of the definition.

\begin{lem}
\label{lem:properties-mes-hol}
Let $\nu$ be $\alpha-$H\"older continuous on $\cE_0'$, then we have 
\begin{itemize}
\item[(a)] $\cE_0 \subset L^1(d\nu)$; 
\item[(b)] if $\mu \leq \nu$, then $\mu$ is also $\alpha-$H\"older continuous on $\cE_0'$;
\item[(b)] $\nu$ is a Radon measure, and $\nu$ vanishes on pluripolar sets in 
$\Omega$.
\end{itemize}
\end{lem}

\begin{proof}
The property $(a)$ is obvious. The property $(b)$ follows from \eqref{eq:measure-holder-b}. Lastly, since $C^{\infty}_c(\Omega)\subset \cE_0 \cap C^0(\bar\Omega) - \cE_0\cap C^0(\bar\Omega)$ (see \cite[Lemma~3.1]{Ce04}) it follows that $\nu$ is a Radon measure. Next, let $K \subset \Omega$ be a compact pluripolar set. Let $G_j$ be a decreasing sequence of compact sets in $\Omega$ satisfying 
$$G_{j+1} \subset \overset{o}{G_j} \quad\mbox{and} \quad \bigcap_{j=1}^\infty G_j = K.$$ 
Let $h_{G_j}$ denote the relatively extremal function of $G_j$. By a theorem of Bedford-Taylor \cite{BT82} we have
\[\notag
	cap(G_j,\Omega) = \int_\Omega (dd^c h_{G_j}^*)^n \quad
	\mbox{and}\quad h_{G_j}^* = -1\quad \mbox{on } \overset{o}{G_j}.
\]
Hence,
\[\notag
\begin{aligned}
	\nu(K) \leq \int_\Omega |h_{G_j}^*| d\nu \leq C \left[cap (G_j, \Omega)\right]^\frac{1-\alpha}{n}  \left(\int_\Omega |h_{G_j}^*| dV_{2n}\right)^\alpha,
\end{aligned}\] 
where the second inequality used the fact that $\nu$ is $\alpha-$H\"older continuous on $\cE_0'.$ Since $h_{G_j}^* \nearrow h_K^* \equiv 0$ in $\Omega$, then the right hand side tends to $0$ as $j\to \infty$. Thus, $\nu$ vanishes on pluripolar sets in $\Omega.$
\end{proof}

Let us state the announced above characterisation for measures in $\cM(\vphi,\Omega)$.
\begin{thm}
\label{thm:general} Let $\mu$ be a positive Borel measure such that $\mu \leq (dd^c\vphi)^n$ for $\vphi \in PSH(\Omega)\cap C^{0,\alpha}(\bar\Omega)$ and $\vphi =0$ on $\d\Omega$, where $0<\alpha \leq 1.$ Assume that $\mu$ has finite total mass. Then, $\mu$ is H\"older continuous on $\cE_0'$ if and only if there exists $u\in PSH(\Omega) \cap C^{0,\alpha'}(\bar\Omega)$, where $0<\alpha'\leq 1$, satisfying
\[\notag
	(dd^c u)^n = \mu, \quad u_{|_{\d\Omega}} = 0.
\]
\end{thm}

\begin{remark}
One should remark that the H\"older exponent of $\mu$ on $\cE_0'$ is often different from the one of $\vphi$. \end{remark}

Let us prepare ingredients to prove the theorem. The following lemma will be its necessary condition. 

\begin{lem}
\label{lem:holder-potential-measure} Let $\vphi \in PSH(\Omega)\cap C^{0,\alpha}(\bar\Omega)$ be such that $\vphi =0$ on $\d\Omega$, where $0<\alpha \leq 1.$ Assume that $$\int_\Omega (dd^c\vphi)^n<+\infty.$$ Then, the measure $\nu:=(dd^c\vphi)^n$ is  H\"older continuous on $\cE_0'$. 
\end{lem}

\begin{remark}\label{rmk:dn-cpt} On a comact K\"ahler manifold $(X,\omega)$ this is an analogue of Dinh-Nguyen \cite[Proposition~4.1]{DN16} for $\omega-$plurisubharmonic functions. The difference is that in the local setting one needs to deal with the boundary terms which do not appear in the compact manifold.  We succeed to estimate these terms by using the Cegrell inequality \eqref{eq:cegrell-ineq}.\end{remark}

\begin{proof} Denote $\beta= dd^c |z|^2$ and
\[\notag
	S_k:= (dd^c\vphi)^{k} \wed \beta^{n-k}.
\]
Our goal is to show that there exists $0<\alpha \leq 1$ such that for $w, v\in \cE_0'$
\[\notag
	\int_\Omega |w-v| S_n \leq C \|w-v\|_{1}^\alpha.
\]
We proceed  by induction over $0\leq k \leq n$. For $k=0$, the statement obviously holds true with $\alpha_0 =1$. Assume that this holds for integers up to $k<n$. For simplicity let us denote
\[\notag
	S:= (dd^c \vphi)^{k}\wed \beta^{n-k-1}.
\]
The induction hypothesis tells us that there is $0<\alpha_k \leq 1$ such that
\[\label{eq:induction-holder}
	\int_\Omega |w-v| S \wed \beta \leq C  \|w-v\|_{1}^{\alpha_k}.
\]
We need to show that there exists $0<\alpha_{k+1} \leq 1$ such that 
\[ \label{eq:ind-est-goal}
\int_\Omega (w-v) dd^c\vphi \wed S \leq C \|w-v\|_1^{\alpha_{k+1}}
\]
for every $u,v \in \cE_0'$ and $u\geq v$ (the general case will follow as in the proof of Lemma~\ref{lem:holder-equivalence}). Indeed, without loss of generality we may assume that $$\|w-v\|_1>0.$$ Otherwise, the inequality will follow directly from the induction step \eqref{eq:induction-holder}.  Let us still write $\vphi$ to be the H\"older continuous extension of $\vphi$ onto a neighbourhood $U$ of $\bar \Omega$. Consider the convolution of $\vphi$ with the standard smooth kernel $\chi$, i.e.,
$\chi \in C^\infty_c(\Omega)$ such that $\chi(z) \geq 0$, $\supp \chi \subset\subset \Omega $ and $\int_{\bC^n}\chi (z) dV_{2n} =1$.
 Then, we observe that
\[\label{eq:observe-holder-a}
	\vphi_\delta(z) - \vphi(z) = \int_{U} [\vphi(z-\delta \zeta) -\vphi(z)] \chi (\zeta) dV_{2n}(\zeta) 
	\leq C \delta^{\alpha},
\]
\[\label{eq:observe-holder-b}
	\left|\frac{\d^2 \vphi_\delta}{\d z_j \d \bar z_k} (z)\right| \leq \frac{C \|\vphi\|_\infty}{\delta^2}.
\]
We first have
\[\label{eq:ind-est}\begin{aligned}
	\int_\Omega (w-v) dd^c \vphi \wed S 
&\leq 	\left| \int_\Omega (w-v) dd^c \vphi_\delta \wed S \right| \\
&\quad	+\left| \int_\Omega (w-v) dd^c (\vphi_\delta -\vphi) \wed S\right| \\
&=: I_1+ I_2.
\end{aligned}\]
It follows from \eqref{eq:induction-holder} and \eqref{eq:observe-holder-b}  that
\[\label{eq:ind-est-a}
\begin{aligned} I_1
&\leq 	\frac{C\|\vphi\|_\infty}{\delta^2} \int_\Omega (w-v) S \wed \beta \\
&\leq 	\frac{C\|\vphi\|_\infty}{\delta^2} \|w-v\|_1^{\alpha_k}.
\end{aligned}\]
The integration by parts in $\cE_0$ gives:
\[\label{eq:ind-est-aa}
\begin{aligned}
	\int_\Omega wdd^c (\vphi_\delta - \vphi) \wed S 
&=	 \int_\Omega (\vphi_\delta - \vphi) dd^c w\wed S \\
&\quad - \int_{\d \Omega} (\vphi_\delta- \vphi) d^c w \wed S.\\
\end{aligned}\]
Notice that $d^c w\wed S$ is a positive measure on $\d\Omega$ with total mass is
\[\label{eq:ind-est-b}
	\int_\Omega dd^c w \wed S <+\infty.
\]
This finiteness of the integral is obtained as follow. By \eqref{eq:def-fct-a}  it is clear that the defining function $\rho  \in \cE_0$ and $dd^c \rho \geq  \beta$. Therefore
\[\notag
	\int_\Omega dd^c w \wed S \leq  \int_{\Omega} dd^c w \wed (dd^c \vphi)^k\wed (dd^c \rho)^{n-k-1}.
\]
The Cegrell inequality~\eqref{eq:cegrell-ineq} gives us that the right hand side is less than
\[\notag
\left(\int_\Omega (dd^cw)^n\right)^\frac{1}{n} \left(\int_\Omega(dd^c\vphi)^n\right)^\frac{k}{n} \left( \int_\Omega (dd^c\rho)^n\right)^\frac{n-k-1}{n} \leq C
\]
as $w \in \cE_0'$ and $\vphi \in\cE_0.$  Notice that here the constant $C$ is independent of $w$. 
Hence, using \eqref{eq:observe-holder-a}, we get  that
\[\label{eq:ind-est-c}
	\left | \int_{\d\Omega} (\vphi_\delta -\vphi) d^c w \wed S \right | \leq C \delta^\alpha \int_\Omega dd^cw \wed S.
\]
Similarly, 
\[\label{eq:ind-est-d}
	\int_\Omega dd^c v\wed S <+\infty,
\]
\[\label{eq:ind-est-dd}
	\left|\int_{\d\Omega} (\vphi_\delta -\vphi) d^cv\wed S\right| \leq C \delta^\alpha \int_\Omega dd^c v\wed S.
\]
Finally, using \eqref{eq:observe-holder-a}, \eqref{eq:ind-est-aa}, \eqref{eq:ind-est-b}, \eqref{eq:ind-est-c}, \eqref{eq:ind-est-d} and \eqref{eq:ind-est-dd}  we are able to finish the estimate 
\[\label{eq:ind-est-e}
\begin{aligned}
I_2 
&\leq 	\left| \int_\Omega(\vphi_\delta -\vphi) dd^c w \wed S \right| + \left| \int_\Omega(\vphi_\delta -\vphi) dd^c v \wed S \right| \\
&\quad + \left|\int_{\d \Omega} (\vphi_\delta -\vphi) d^c w \wed S\right| + \left|\int_{\d \Omega} (\vphi_\delta -\vphi) d^c v \wed S\right|\\
&\leq 	C \delta^\alpha \left( \int_\Omega dd^c w \wed S + \int_\Omega dd^c v \wed S\right) \\
&\leq		C \delta^\alpha.
\end{aligned}\]
Thanks to \eqref{eq:ind-est-b} and \eqref{eq:ind-est-d} we have
\[\notag\begin{aligned}
	\left| \int_\Omega (w-v) dd^c \vphi \wed S \right| 
&\leq		\left|\int_\Omega \vphi dd^c w \wed S\right| + \left|\int_\Omega \vphi dd^c v\wed S\right| \\
&\leq C \|\vphi\|_\infty.
\end{aligned}\]
From this inequality we can assume that $0<\|w-v\|_1 <1/100$.  
Combining \eqref{eq:ind-est}, \eqref{eq:ind-est-a} and \eqref{eq:ind-est-e} we have
\[\notag
	\int_\Omega (w-v) dd^c \vphi\wed S \leq \frac{C\|\vphi\|_\infty}{\delta^2} \|w-v\|_1^{\alpha_k} + C \delta^\alpha.
\]
Therefore, the proof of \eqref{eq:ind-est-goal} is completed if we choose $$\delta = \|w-v\|_1^\frac{\alpha_k}{3}, \quad \alpha_{k+1}= \frac{\alpha \alpha_k}{3}.$$
Thus, the lemma is proven.
\end{proof}

The measures which are H\"older continuous on $\cE_0'$ satisfy the volume-capacity inequality, which is the content of the next proposition. This inequality plays a crucial role in deriving the {\em apriori} uniform estimate and stability estimates for the solutions to the Monge-Amp\`ere equation.

\begin{prop}
\label{prop:analogue-dn}
Assume that $\nu$ is $\alpha-$H\"older continuous on $\cE_0'$ with finite total mass on $\Omega$. 
Then, there exist uniform constants $\alpha_1>0$ and $C>0$  such that for every compact set $K\subset \Omega$
\[\label{eq:vol-cap}
	\nu (K) \leq C \exp\left(\frac{-\alpha_1}{\left[cap(K,\Omega)\right]^\frac{1}{n}}\right).
\]
\end{prop}

\begin{proof}  First, we prove that for  $v\in PSH\cap L^\infty(\Omega)$ and $\lim_{z\to \d \Omega} v(z) =0$ satisfying $\int_\Omega (dd^c v)^n \leq 1$, there exist uniform constants $\alpha_1>0$ (small) and $C>0$ such that
\[\label{eq:vol-cap2}
	\nu (v< -s) \leq C e^{-\alpha_1 s} \quad \forall s>0.
\]
Indeed, put $v_s:= \max\{v, -s\}$. Then
\[\label{eq:vol-cap-a}
	\nu(v < -s -1) = \int_{\{v<-s-1\}} d\nu \leq \int_\Omega |v_s -v| d\nu.
\]
Since $v_s, v \in \cE_0'$ and $\nu$ is $\alpha-$H\"older continuous on $\cE_0'$, the right hand side is bounded by 
\[\label{eq:vol-cap-b}
	C  \left(\int_\Omega |v_s-v|dV_{2n}\right)^\alpha.
\]
Since $0 \leq v_s - v \leq {\bf 1}_{\{v<-s\}} |v| \leq {\bf 1}_{\{v<-s\}} e^{-\alpha_2v}/\alpha_2$ for $\alpha_2>0$ (to be determined later), we have
\[\label{eq:vol-cap-c}\begin{aligned}
	\int_{\Omega} |v_s -v| dV_{2n}
&\leq 	\frac{1}{\alpha_2}\int_{\{v<-s\}} e^{-\alpha_2 v} dV_{2n} \\
&\leq 	\frac{1}{\alpha_2}\int_\Omega e^{-\alpha_2 v - \alpha_2(v+s)} dV_{2n} \\
&\leq 	\frac{e^{-\alpha_2 s}}{\alpha_2} \int_\Omega e^{-2 \alpha_2v} dV_{2n}.
\end{aligned}\]
The last integral is uniformly bounded for $2\alpha_2>0$ small enough thanks to \cite[Lemma 4.1]{ko05}. Combining \eqref{eq:vol-cap-a}, \eqref{eq:vol-cap-b} and \eqref{eq:vol-cap-c}
we get that 
\[\notag
	\nu(v <-s-1) \leq C e^{-\alpha\alpha_2 s}.
\]
Hence, we obtained \eqref{eq:vol-cap2} for $\alpha_1 = \alpha\alpha_2$ with $\alpha_2>0$ small enough.

To finish the proof of the proposition we use an argument which is inspired by the proofs in \cite{ACKPZ09}. Let $K\subset \Omega$ be compact. Since $\nu$ vanishes on pluripolar sets (Lemma~\ref{lem:properties-mes-hol}) we may assume that $K$ is non-pluripolar. Let $h_K^*$ be the relative extremal function of  $K$ with respect to $\Omega$.  
Since $K \subset \Omega$ is compact, it is well-known that 
\[\notag
	\lim_{\zeta \to \d \Omega} h_{K}^*(\zeta)=0.
\]
By \cite[Proposition~5.3]{BT82} we have
\[\notag
	\tau^n:= cap(K, \Omega) = \int_{\Omega} (dd^ch_K^*)^n >0.
\]
Let $0<x<1.$ Since the function $w:= \frac{h_K^*}{\tau}$ satisfies assumptions of the inequality \eqref{eq:vol-cap2}, we have
\[\notag
	\nu (h_K^* < -1 +x) = \nu \left(w < \frac{-1+x}{\tau} \right) \leq C \exp\left({-\frac{\alpha_1(1-x)}{\tau}}\right).
\]
Let $x\to 0^+$, we obtain
\[\label{eq:cap-ineq1}
	\nu(h_K^* \leq -1) \leq C \exp\left({\frac{-\alpha_1}{\left[cap(K,\Omega)\right]^\frac{1}{n}}}\right).
\]
Since $h_K = h_K^*$ outside a pluripolar set, we have
\[\label{eq:cap-ineq2}
	\nu(K) \leq \nu(h_K =-1) = \nu(h_K^* =-1) \leq \nu(h_K^* \leq -1).
\]
We combine \eqref{eq:cap-ineq1} and \eqref{eq:cap-ineq2} to finish the proof.
\end{proof}

We can follow the proofs in  \cite{EGZ09}, \cite{GKZ08}  and \cite{ko96}, to derive the stability estimate for the measures which are well dominated by capacity.

\begin{prop}
\label{prop:stability} 					
Suppose that $\mu = (dd^cu)^n$ for $u \in PSH(\Omega) \cap C^0(\bar\Omega)$ and  $\mu$ satisfies the inequality \eqref{eq:vol-cap}. Let $v \in PSH(\Omega) \cap C^0(\bar \Omega)$ be such that $u = v$ on $\d \Omega$.
Then, there exist constants $C>0$ and $0<\alpha_2<1$ independent of $v$ such that
\[
	\sup_{\Omega} (v-u) \leq C \left(\int_\Omega \max\{v-u,0\} d\mu\right)^{\alpha_2}.
\]
\end{prop}

\begin{proof} It readily follows from the one of \cite[Theorem~1.1]{GKZ08} with obvious adjustments.
\end{proof}

We proceed to the proof of Theorem~\ref{thm:general}. 
Let $\mu \in \cM(\vphi, \Omega)$ with $\mu(\Omega) <+\infty.$ The theorem asserts that $\mu$ is H\"older continuous on $\cE_0'$ if and only there exists $u\in PSH(\Omega)\cap C^{0,\alpha'}(\bar\Omega)$, $0<\alpha'\leq 1$, solving
$$	
	(dd^cu)^n = \mu, \quad u_{|_{\d\Omega}} =0.
$$

\medskip

First, we see that the sufficient condition follows from Lemma~\ref{lem:holder-potential-measure} applied for $\vphi:=u$. It remains to prove the necessary condition.  By Proposition~\ref{prop:analogue-dn}, the measure $\mu$ satisfies the volume-capacity inequality \eqref{eq:vol-cap}. Therefore, using Ko\l odziej's result \cite[Theorem~5.9]{ko05} we solve $u \in PSH(\Omega) \cap C^0(\bar\Omega)$, 
\[\label{eq:ma-necessary} 
	(dd^cu)^n = \mu, \quad u_{|_{\d\Omega}} =0.
\]
The existence of the H\"older continuous subsolution assures that the solution $u$ is H\"older continuous on the boundary. To see this, we get by the comparison principle \cite{BT82} that
\[\label{eq:boundary-est-a}
	\vphi \leq u \leq 0 \quad \mbox{on } \Omega.
\]
Let us denote for $\delta>0$ small
\[
	\Omega_\delta := \left\{z \in \Omega : dist(z, \d \Omega) > \delta\right\};
\]
and for $z\in \Omega_{\delta}$ we define
\[\label{eq:boundary-est-b}
	u_\delta(z) := \sup_{|\zeta| \leq \delta} u(z+ \zeta),
\]
\[\label{eq:boundary-est-bb}
	\hat u_\delta(z) 
:= \frac{1}{\sigma_{2n} \delta^{2n}} \int_{|\zeta| \leq \delta} u(z + \zeta) dV_{2n}(\zeta),
\]
where $\sigma_{2n}$ is the volume of the unit ball.
\begin{lem}
\label{lem:holder-boundary}
There exist constants $c_0 = c_0(\vphi)>0$ and $\delta_0>0$ small such that for any $0<\delta <\delta_0$,
\[
	u_\delta \leq u(z) + c_0 \delta^{\alpha} 
\]
for every $z\in \d \Omega_{\delta}$, where $\alpha$ is the H\"older exponent of $\vphi.$
\end{lem}

\begin{proof}
Fix a point $z\in \d\Omega_\delta$.
It follows from \eqref{eq:boundary-est-a} and \eqref{eq:boundary-est-b} that 
\begin{align*}
	u_\delta (z) - u(z) \leq  -\vphi(z).
\end{align*}
Choose $\zeta_0 \in \bC^n$, $\|\zeta_0\| = \delta$, such that
$z+\zeta_0 \in \d \Omega$. Then, $\vphi(z+\zeta_0)=0$, and
\begin{align*}
 - \vphi (z)
&	=  \vphi(z+\zeta_0) - \vphi(z) \\
&	\leq  c_0 \; \delta^\alpha 
\end{align*}
for $0< \delta < \delta_0$ as $\vphi \in C^{0,\alpha}(\bar\Omega)$.
\end{proof}

The following result follows essentially from \cite[Theorem~3.3, Theorem~3.4]{BKPZ16}.

\begin{lem}
\label{lem:mass-growth}  Fix $0< \alpha_4 <1$. Then, we have for every $\delta>0$ small,
\[
	\int_{\Omega_{\delta}} |\hat u_\delta-u| dV_{2n} \leq  C \delta^{\alpha_4}.
\]
\end{lem}

\begin{proof}
First, we know from the classical Jensen formula (see e.g \cite[Lemma~4.3]{GKZ08}) that
\[\label{eq:jensen-ineq}
	\int_{\Omega_{2\delta}} |\hat u_\delta -u| dV_{2n} \leq C \delta^2 \int_{\Omega_{\delta}} \Delta u(z).
\]
Since both $\rho$ and $- dist (z, \d\Omega)$ are  defining functions of the $C^2$ bounded  domain,  we have $-\rho(z) \geq c_1 \;dist(z, \d\Omega)$ for some uniform constant $c_1 >0$ depending on $\rho$ and $\Omega$. Hence, $$\Omega_{\delta} \subset \{\rho(z) < - c_1\delta\}.$$
Put $2\delta_1= c_1\delta$. Then,
\[\label{eq:jensen-ineq-a0}
	\int_{\Omega_\delta} \Delta u(z) \leq \int_{\{\rho < -2\delta_1\}} \Delta u(z).
\]
We may also assume that $\{\rho<-\delta_1\}$ has the smooth boundary. Then, we have for a fixed $0< \alpha_4 <1$,
\[\label{eq:jensen-eq-a}\begin{aligned}
&	\int_{\{\rho<-2\delta_1\}} \Delta u(z) \\
&\leq 	\delta_1^{-(2-\alpha_4)}\int_{\{\rho<-\delta_1\}} (-\rho -\delta_1)^{2-\alpha_4} \Delta u(z) \\
&= 		c_n \;\delta_1^{-(2-\alpha_4)}  \int_{\{\rho<-\delta_1\}} (-\rho -\delta_1)^{2-\alpha_4} dd^c u \wed \beta^{n-1} \\
&= 		c_n \;\delta_1^{-(2-\alpha_4)}  \int_{\{\rho<-\delta_1\}} u  dd^c (-\rho -\delta_1)^{2-\alpha_4} \wed \beta^{n-1} \\
&\quad + c_n \;\delta_1^{-(2-\alpha_4)} \int_{\{\rho = -\delta_1\}} (-u) d^c(-\rho-\delta_1)^{2-\alpha_4} \wed \beta^{n-1},
\end{aligned}\]
where we used the integration by parts for the third equality, and $c_n$ is a constant depends only on $n$. 

Next, we are going to get a uniform bound for the last two integrals in \eqref{eq:jensen-eq-a}. We compute
\[\label{eq:jensen-eq-b}
\begin{aligned}
dd^c (-\rho-\delta_1)^{2-\alpha_4} 
&= 		(2-\alpha_4) (1-\alpha_4) (-\rho-\delta_1)^{-\alpha_4} d\rho\wed d^c\rho \\
&\quad - (2-\alpha_4) (-\rho-\delta_1)^{1-\alpha_4} dd^c\rho.
\end{aligned}\]
It is clear that
\[\label{eq:jensen-eq-c}
	\int_{\{\rho<-\delta_1\}} |u| (-\rho-\delta_1)^{-\alpha_4} d\rho\wed d^c\rho \wed \beta^{n-1} \leq C \|u\|_\infty \int_{\{\rho<0\}} |\rho|^{-\alpha_4} dV_{2n}.
\]
The integral on the right hand converges as $0<\alpha_4<1$. Similarly, 
\[\label{eq:jensen-eq-d}
	\int_{\{\rho<-\delta_1\}} |u|(-\rho-\delta_1)^{1-\alpha_4} dd^c \rho \wed \beta^{n-1} 
	\leq C \|u\|_\infty \|\rho\|_\infty \int_{\{\rho<0\}} dV_{2n}.
\]
Therefore, by \eqref{eq:jensen-eq-b}, \eqref{eq:jensen-eq-c} and \eqref{eq:jensen-eq-d} the first integral 
\[\label{eq:jensen-eq-e}
	\int_{\{\rho<-\delta_1\}} u  dd^c (-\rho -\delta_1)^{2-\alpha_4} \wed \beta^{n-1} 
	\leq C.
\]
We continue with the second one.
\[\label{eq:jensen-eq-f}\begin{aligned}
&	\int_{\{\rho=-\delta_1\}} (-u) d^c (-\rho-\delta_1)^{2-\alpha_4} \wed \beta^{n-1} \\
&= 	(2-\alpha_4)	\int_{\{\rho=-\delta_1\}}  (-u)  (-\rho-\delta_1)^{1-\alpha_4} d^c\rho \wed \beta^{n-1}\\
&\leq 	C \|u\|_\infty \|\rho\|_\infty \int_{\{\rho=-\delta_1\}} d^c\rho \wed \beta^{n-1}\\
&\leq		C\|u\|_\infty \|\rho\|_\infty \int_{\{\rho<0\}} dd^c\rho \wed \beta^{n-1}.
\end{aligned}\]
Combining \eqref{eq:jensen-ineq-a0}, \eqref{eq:jensen-eq-a}, \eqref{eq:jensen-eq-e} \eqref{eq:jensen-eq-f}, and then substituting $\delta = 2\delta_1/c_1$ we conclude that 
\[\notag
	\int_{\Omega_{2\delta}} \Delta u(z) \leq C \delta^{-(2-\alpha_4)}.
\]
Hence, using this and \eqref{eq:jensen-ineq} we get that
\[\notag\begin{aligned}
	\int_{\Omega_\delta} |u_\delta -u| dV_{2n} 
&\leq		\int_{\Omega_{2\delta}} |u_\delta -u| dV_{2n} + 2 \|u\|_{\infty}\int_{\Omega_{\delta} \setminus \Omega_{2\delta}} dV_{2n} \\
&\leq		C \delta^{\alpha_4} + C\delta.
\end{aligned}\]
We thus completed the proof of the lemma. 
\end{proof}

We  are in the position to complete the theorem.
 
\begin{proof}[End of the proof of Theorem~\ref{thm:general}] It follows from Lemma~\ref{lem:holder-boundary} and $\hat u_\delta \leq u_\delta$ that
\[ \label{eq:extend-solution}
\tilde u = 
\begin{cases} 
	\max\{\hat u_\delta - c_0 \delta^\alpha, u\} \quad 
	&\mbox{ on } \Omega_{\delta},\\
	u  \quad &\mbox{ on } \Omega\setminus \Omega_{\delta},
\end{cases}
\]
belongs to $PSH(\Omega) \cap C^0(\bar \Omega)$. Notice that 
\[\notag
	\int_\Omega (dd^c\tilde u)^n \leq \int_\Omega (dd^cu)^n = \mu(\Omega)<+\infty.
\]
By applying Proposition~\ref{prop:stability} for $\tilde u$ and $u$, there is $0<\alpha_2 \leq 1$ such that 
\[\label{eq:sup-l1}
\begin{aligned}
	\sup_\Omega (\tilde u -u) 
&\leq 	C \left(\int_\Omega\max\{\tilde u -u,0\} d\mu\right)^{\alpha_2} \\
&\leq 	C \left(\int_\Omega|\tilde u -u| d\mu\right)^{\alpha_2}.
\end{aligned}\]
Since $\mu \leq (dd^c\vphi)^n$, it follows from Lemma~\ref{lem:holder-potential-measure}  that $\mu$ is H\"older continuous on $\cE_0'$.  Moreover, 
$$\frac{u}{c_2}, 
\frac{\tilde u}{c_2} \in \cE_0',
$$
where $c_2 = [\mu(\Omega)]^{1/n}$. Therefore, there is $0<\alpha_3 \leq 1$ such that
\[\notag
	\int_\Omega |\tilde u - u| d\mu \leq C \left[\mu(\Omega)\right]^{\frac{1-\alpha_3}{n}}  \|\tilde u -u\|_1^{\alpha_3}. 
\]
By Lemma~\ref{lem:mass-growth} and \eqref{eq:extend-solution} the right hand side is controlled by 
\[\label{eq:l1-measure-lebesgue} 
	C \delta^{\alpha_3\alpha_4}
\] 
for a fixed $0<\alpha_4<1.$
Hence, using Lemma~\ref{lem:holder-boundary}, \eqref{eq:sup-l1} and \eqref{eq:l1-measure-lebesgue},
\[\notag\begin{aligned}
	\sup_{\Omega_{\delta}} (\hat u_\delta - u) 
&\leq		 \sup_\Omega (\tilde u - u) + c_0\delta^\alpha \\
&\leq		C \delta^{\alpha_5},
\end{aligned}\]
where $\alpha_5= \min\{\alpha, \alpha_2\alpha_3\alpha_4\}$. Thanks to \cite[Lemma 4.2]{GKZ08} we infer that
\[\notag
	\sup_{\Omega_\delta} (u_\delta -u) \leq C \delta^{\alpha_5}
\]
for $0<\delta < \delta_0$. Thus the proof of the H\"older continuity of $u$ in $\bar\Omega$ is completed.
\end{proof}

\medskip

Let us now complete the proof of the main theorem.

\begin{proof}[Proof of Theorem~\ref{thm:main}] Since $\mu \in \cM(\vphi,\Omega)$ and $\int_\Omega (dd^c\vphi)^n <+\infty$, it follows from Lemma~\ref{lem:properties-mes-hol} and Lemma~\ref{lem:holder-potential-measure} that $\mu$ is H\"older continuous on $\cE_0'$. By Theorem~\ref{thm:general} there exists $u\in PSH(\Omega)\cap C^{0,\alpha'}(\bar\Omega)$, $0<\alpha' \leq 1$, solving
$$
	(dd^cu)^n = d\mu, \quad u_{|_{\d\Omega}} =0.
$$
Thus, the proof is completed.
\end{proof}

Next, we give the following simple consequence of the main theorem. This result is also partially proved by another way in \cite[Theorem~3.6]{Cha15b}.

\begin{cor}\label{cor:simple-consequence}
Let $\mu\in \cM(\vphi,\Omega)$. Assume that $\mu$ has  compact support in $\Omega$. Then Dirichlet problem~\eqref{eq:zeriahi-prob} is solvable.
\end{cor}

\begin{proof} Let $\rho$ be the defining function as in \eqref{eq:def-fct-a}. Assume that 
$K:=\supp \mu \subset\subset \Omega
$  and we fix a small constant $\delta>0$.
There exists $A>0$ large enough such that $A\rho \leq \vphi - \delta$ in the neighbourhood $\Omega' \subset\subset \Omega$ of $K$. Set
\[ \notag
\tilde\vphi (z) = 
\max\{\vphi (z) - \delta, A \rho(z)\}.
\]
We easily see that $\tilde\vphi = \vphi$ on $\Omega'$ and $\tilde\vphi = A \rho$ near the boundary $\d\Omega$, and $\tilde\vphi$ is H\"older continuous in $\bar\Omega$ with the same exponent of $\vphi$. Therefore, 
\[\notag
	\mu \leq (dd^c\tilde\vphi)^n \quad\mbox{and}\quad\int_\Omega (dd^c \tilde\vphi)^n <+\infty.
\]
The finiteness of the integral is followed from the Chern-Levine-Nirenberg inequality and $C^2$-smoothness near the boundary of $\tilde\vphi$. 
Using Theorem~\ref{thm:main} we can solve the Dirichlet problem \eqref{eq:zeriahi-prob} for $\mu$.
\end{proof}

It is obvious that the Lebesgue measure $dV_{2n}$ is $1$-H\"older continuous on $\cE_0'$. Then, we have the following general result which contains previous results in \cite{BKPZ16}, \cite{Cha15a}, \cite{GKZ08} for the zero boundary case.

\begin{cor}
\label{cor:lp-density} Let $\nu$ be H\"older continuous on $\cE_0'$. Let $0 \leq f \in L^p(\Omega, d\nu)$, $p>1$.  Then, $f d\nu$ is  H\"older continuous on $\cE_0'$. In particular, the Dirichlet problem \eqref{eq:zeriahi-prob} admits a unique solution for the measure $\mu = fd\nu$. 
\end{cor}

\begin{proof}
The case $n=1$ is classical, so we assume that $n\geq 2$. It is sufficient  to show that $f d\nu$ is H\"older continuous on $\cE_0'$. Assume that $u, v \in \cE_0'$ and $u\geq v$. Consider the triples
\[\notag
	(q_1, q_2,q_3) = \left((n-1)p+1, \frac{(n-1)p+1}{n-1}, \frac{(n-1)p+1}{(n-1)(p-1)}\right)
\]
and 
\[\notag
\left( \left[(u-v)^nf\right]^\frac{1}{q_1}, f^\frac{p}{q_2}, (u-v)^\frac{1}{q_3}\right).
\]
Then, the generalised H\"older inequality for 
\[\notag
	\frac{1}{q_1} + \frac{1}{q_2} + \frac{1}{q_3} = 1
\]
gives us that
\[\label{eq:lp-case-a}
\int_\Omega (u-v) f d\nu \leq  \left(\int_\Omega(u-v)^n fd\nu\right)^{1/q_1} \|f\|_{L^p(d\nu)}^{p/q_2} \|u-v\|_{L^1(d\nu)}^{1/q_3}. 
\]
Since $\nu$ is H\"older continuous on $\cE_0'$, there is $0<\alpha \leq 1$ such that
\[\label{eq:lp-case-a1}
	\int_\Omega |u-v| d\nu \leq \|u-v\|_1^\alpha.
\]
Therefore, to end the proof, we need to verify that the first factor of the right hand side in \eqref{eq:lp-case-a} is uniformly bounded. Indeed, by the H\"older inequality we have
\[\notag
	\int_K f d\nu \leq \|f\|_{L^{p}(d\nu)}^\frac{1}{p} \left[\nu(K)\right]^\frac{1}{q},
\]
where $1/p+1/q =1$. Hence, $fd\nu$ satisfies the volume-capacity inequality \eqref{eq:vol-cap}. 
By Ko\l odziej's theorem \cite[Theorem~5.9]{ko05} there exists   $\psi \in PSH(\Omega)\cap C^0(\bar\Omega)$ satisfying
$
	(dd^c \psi)^n = f d\nu, \quad \psi_{|_{\d\Omega}} =0
$ 
and 
\[\label{eq:lp-case-b} \|\psi\|_\infty \leq C\left( \|f\|_{L^p(d\nu)}, \Omega, \nu\right).
\]
By an estimate of B\l ocki \cite{Bl93} we have
\[\label{eq:lp-case-c}
	\int_\Omega (u-v)^n (dd^c\psi)^n \leq  n! \|\psi\|_\infty^n \int_\Omega (dd^cv)^n.
\]
The right hand side under controlled as  $v \in \cE_0'$. It follows from  \eqref{eq:lp-case-a},\eqref{eq:lp-case-a1}, \eqref{eq:lp-case-b} and \eqref{eq:lp-case-c} that
\[\notag
	\int_\Omega (u-v)f d\nu\leq C \|u-v\|_1^\frac{\alpha}{q_3}.
\]
Thus, the corollary follows.
\end{proof}

Finally, we emphasise that several interesting examples of positive Borel measures for which one can solve the Dirichlet problem~\eqref{eq:zeriahi-prob} are given in Charabati \cite{Cha15b} (see also \cite{hiep}, \cite{viet16}). In particular, the class of Hausdorff-Riesz measures of order $2n-2+\vepsilon$  with $0<\vepsilon \leq 2$ in $\Omega$. By Theorem~\ref{thm:general} such measures are H\"older continuous on $\cE_0'.$  On the other hand we can prove directly this result.

\begin{lem}
\label{lem:riesz-hausdorff-measure}
Let $\mu$ be a Hausdorff-Riesz measure of order $2n-2 +\vepsilon$ with $0< \vepsilon \leq 2$ in $\Omega$, i.e., for every $B(z,\delta)\subset \Omega$ 
\[\label{eq:hr-a}
	\mu(B(z,\delta)) \leq C \delta^{2n-2+\vepsilon},
\]
where $0<\delta <\delta_0$ small. Assume that $\mu$ has finite total mass. Then, $\mu$ is H\"older continuous on $\cE_0'$.
\end{lem}

\begin{proof} By classical potential theory we solve $\vphi \in SH(\Omega)\cap C^{0,\alpha}(\bar\Omega)$ satisfying
\[\notag
	\Delta \vphi = d\mu, \quad  \vphi_{|_{\d\Omega}} =0.
\]
The H\"older continuity follows from \eqref{eq:hr-a} (see e.g. \cite[Remark~4.2]{DDGKPZ14}). Let us still write $w$ for its H\"older continuous extension onto a neighbourhood $U$ of $\bar\Omega$. Let $w, v \in \cE_0'$ and $w\geq v$. We need to show that
\[\notag
	\int_{\Omega} (w-v) \Delta \vphi  \leq C \|w-v\|_1^{\alpha'}
\]
for some $0<\alpha' \leq 1.$ We can write 
$
	\Delta \vphi =  n dd^c \vphi \wed \beta^{n-1}
$
in the sense of measures. Define $\vphi_\delta$ is the convolution of $\vphi$ with the standard smooth kernel, then we have as in \eqref{eq:observe-holder-a} and \eqref{eq:observe-holder-b} that 
\begin{align*}
	|\vphi_\delta (z) - \vphi (z)| \leq C \delta^\alpha, \\
	\left|\Delta \vphi_\delta (z)\right| \leq \frac{C \|\vphi\|_\infty}{\delta^2}.	
\end{align*}
From this point the proof goes the same as the one in Lemma~\ref{lem:holder-potential-measure} with the closed positive current $S = \beta^{n-1}$.
\end{proof}

\bigskip



\end{document}